\documentclass[12pt,a4paper]{article}
\usepackage{amsmath,amssymb,amsthm}
\usepackage{amssymb,latexsym, bbm,comment}
\usepackage{graphicx}
\usepackage{makeidx}
\newtheorem{theorem}{Theorem}[section]
\newtheorem{lemma}[theorem]{Lemma}
\newtheorem{prop}[theorem]{Proposition}
\newtheorem{cor}[theorem]{Corollary}
\numberwithin{equation}{section}

\newcommand{\R}{\mathbb{R}}
\newcommand{\Rd}{\mathbb{R}^d}
\newcommand{\nN}{n \in \mathbb{N}}
\newcommand{\N}{\mathbb{N}}
\newcommand{\C}{\mathbb{C}}

\newcommand{\ds}{\displaystyle}
\newcommand{\g}{\mathfrak{g}}

\newcommand{\tr}{\mbox{tr}}

\newcommand{\bean}{\begin{eqnarray*}}
\newcommand{\eean}{\end{eqnarray*}}

\newcommand{\la}{\langle}
\newcommand{\ra}{\rangle}

\newcommand{\Z}{\mathbb{Z}}

\newcommand{\G}{\widehat{G}}

\begin{document}

\date{}

\title{The Positive Maximum Principle on Lie Groups}

\author{David Applebaum, Trang Le Ngan\\ School of Mathematics and Statistics,\\ University of
Sheffield,\\ Hicks Building, Hounsfield Road,\\ Sheffield,
England, S3 7RH\\ ~~~~~~~\\e-mail: D.Applebaum@sheffield.ac.uk, tlengan1@sheffield.ac.uk}

\maketitle

\begin{abstract}
We extend a classical theorem of Courr\`{e}ge to Lie groups in a global setting, thus characterising all linear operators on the space of smooth functions of compact support that satisfy the positive maximum principle. We show that these are L\'{e}vy type operators (with variable characteristics), and pseudo--differential operators when the group is compact. If the characteristics are constant, then the operator is the generator of the contraction semigroup associated to a convolution semigroup of sub--probability measures.

\vspace{5pt}


\vspace{5pt}

\begin{center}


\end{center}

\vspace{5pt}

\begin{center}

{\it We dedicate this paper to the memory of Herbert Heyer 1936--2018.}

\end{center}

\end{abstract}

\section{Introduction}

In 1965/6, Courr\`{e}ge \cite{Courr}, building on work of von Waldenfels \cite{Wal}, classified all linear operators in Euclidean space that satisfy the positive maximum principle(PMP). He first showed that these were what we now call {\it L\'{e}vy type operators}, in that they are the sum of a second order elliptic operator and a non--local term expressed as integration against a kernel. Secondly he showed that all such operators may be realised as pseudo--differential operators, but because of the non-local term, the symbol does not in general have the smoothness properties that are typically encountered when meeting pseudo--differential operators, say in the study of pde's, or in index theory. Nonetheless, in recent years  Courr\`{e}ge's theorem has found a great deal of application in the study of Feller--Markov processes, due mainly to the work of N.Jacob, R.Schilling and their collaborators (see e.g. \cite{Jac1, Jac4, BSW}). The key observation is that the generator of a Feller--Markov semigroup having a sufficiently rich domain satisfies the PMP, and so will be both a L\'{e}vy type operator, and a pseudo--differential operator, whose symbol is a rich source of probabilistic information about the behavior of the process. The aim of the current paper is to extend this theory to Lie groups.

In fact, Courr\`{e}ge \cite{Courr1} has partially extended his results to obtain the form of generators of Feller--Markov processes on compact manifolds, and these were further extended to more general manifolds in \cite{BCP}; but the results are given in terms of local co--ordinates, and are not presented in a form that is easy to use. Our approach is more powerful, at least in the Lie group context, in that we are able to classify all operators satisfying the PMP, and our results are global, with all differential operators being expressed in terms of a Lie algebra basis. A particular class of Feller--Markov processes of particular importance are the L\'{e}vy processes \cite{Liao}, and these are essentially in one--to--one correspondence with convolution semigroups of probability measures on $G$. The associated semigroup generator was classified by Hunt \cite{Hu} in 1956. The Hunt generator, as would be expected (see e.g. section 5 of \cite{App2}), has essentially the same structure as the L\'{e}vy--type operators that we obtain, but the functions and kernels that parameterise the linear operator are constant in this special case.

Of course pseudo--differential operators have been considered on manifolds since the theory was first developed; but in a local sense using local co--ordinates. In  Chapter 10 of \cite{RT} Ruzhansky and Turunen have created a global theory of pseudo--differential operators on compact Lie groups, making use of Peter--Weyl theory (see also \cite{RTW} for further developments). So the symbol of the operator, in that sense, is a matrix and not a scalar. This theory was applied to Hunt semigroups by one of us in \cite{App3} and some applications to Feller semigroups were developed in \cite{App2}. In the present paper we complete the programme of extending the results of \cite{Courr} to (compact) Lie groups, by showing that all linear operators that satisfy the PMP may be expressed as pseudo--differential operators within a suitable generalisation of the Ruzhansky--Turunen theory. As we might expect, the symbol has the form of a generalised characteristic exponent as occurs in the L\'{e}vy--Khintchine formula for convolution semigroups (see section 5.5. in \cite{App1}, pp.144--6), but with variable characteristics.

We do not follow Courr\`{e}ge's original approach in characterising linear operators satisfying the PMP. Instead we adopt the very slick approach that is given in  Hoh \cite{Hoh} (see p.47 of \cite{BSW} for further historic references), where the problem is reduced to studying associated linear functionals that turn out to be distributions of order $2$. This requires us to make use of a global theory of distributions on Lie groups, as developed by Ehrenpreis \cite{Ehr}. The drawback of Hoh's approach is that it doesn't fit naturally into the framework of the Banach space comprising continuous functions that vanish at infinity, which is natural for studying Feller semigroups. However we are able to find sufficient conditions for the operator of interest to act in that context.

For Euclidean spaces, Courr\`{e}ge's results have led to significant progress in the study of Feller processes, not least through the solution of the martingale problem (see \cite{Hoh, Hoh1} and Chapter 4 of \cite{Jac4}). For other results in this area see the survey \cite{BSW}. Our expectation is that the work of the current paper will open up the possibility of pursuing a similar programme for processes on Lie groups.

The organisation of the paper is as follows. We first give  a short account of distributions on Lie groups, which is sufficient for our purposes. Then we prove the generalised Courr\`{e}ge theorem. Next we examine Hunt's theorem from within this context. Finally we specialise to compact Lie groups and obtain the pseudo--differential operator representation. More detailed accounts of some of the results herein can be found in \cite{Lili}. A forthcoming paper will deal with extensions of many of these results to symmetric spaces \cite{AT}. For more general manifolds, the best results currently available are in \cite{BCP}.

\vspace{5pt}

{\bf Notation.} Throughout this paper, $G$ is a Lie group having neutral element $e$, dimension $d$ and Lie algebra $\g$. The exponential map from $\g$ to $G$ will be denoted as $\exp$. We denote by ${\mathcal F}(G)$, the space of all real--valued functions on $G$, $C_{0}(G)$ the Banach space (with respect to the supremum norm $||\cdot||_{\infty}$) of all real--valued, continuous functions on $G$ that vanish at infinity, and $C_{c}^{\infty}(G)$ the dense linear manifold in $C_{0}(G)$ of smooth functions with compact support. The Borel $\sigma$--algebra of $G$ is denoted as ${\mathcal B}(G)$. For each $g \in G, l_{g}$ is the operation of left translation by $G$ on itself defined by $l_{g}(h) = gh$. The pullback to $C_{0}(G)$ is denoted $L_{g}$, so for all $g \in G, f \in C_{0}(G), L_{g}f = f \circ l_{g}$. Note that $l_{g}$ is a diffeomorphism and $L_{g}$ is a Banach space (linear) automorphism.

\noindent The trace of a real or complex $d \times d$ matrix $A$ is written tr$(A)$. The space of all complex $d \times d$ matrices is denoted by $M_{d}(\C)$.

\section{Distributions on Lie Groups}

We will need some facts about distributions on Lie groups. Of course Lie groups are $C^{\infty}$--manifolds, and the standard approach to defining distributions therein is to make use of local co--ordinate systems (see e.g. section 6.3 of \cite{Hor}). We wish to employ the global structure on the Lie group, and in this respect we follow Ehrenpreis \cite{Ehr}. We do not seek to do more in this direction than develop the few tools that we will need for our purposes.

We fix once and for all a basis $\{X_{1}, \ldots, X_{d}\}$ of $\g$. Let $\alpha = (\alpha_{1}, \ldots, \alpha_{N})$ be a multi--index where $\alpha_{i} \in \Z_{+}$, for $i = 1, \ldots N$, and define $|\alpha| = \alpha_{1} + \cdots + \alpha_{N}$. We are going to introduce the (seemingly ambiguous) notation $X^{\alpha}$, with the understanding that we will only employ it when we sum over a finite subset of $\alpha$'s, and that the sum will contain all such distinct non--commuting (in general) combinations of products of the $X_{i}$'s (acting as differential operators on $C^{\infty}(G)$). So for example, for all $f \in C^{\infty}(G)$, in the case $d = 2$,
$$ \sum_{|\alpha| \leq 2}X^{\alpha}f = f + X_{1}f + X_{2}f + X_{1}^{2}f + X_{1}X_{2}f + X_{2}X_{1}f + X_{2}^{2}f.$$

A {\it distribution} $P$ on $G$ is a real--valued linear functional defined on $C_{c}^{\infty}(G)$ such that for every compact set $K$ contained in $G$, there exists $k \in \Z_{+}, C > 0$ so that for all $f \in C_{c}^{\infty}(K)$,
\begin{equation} \label{dist1}
|Pf| \leq C \sum_{|\alpha| \leq k}||X^{\alpha}f||_{\infty},
\end{equation}
where the sum on the right hand side of (\ref{dist1}) is a convenient shorthand for
$$ ||f||_{\infty} + \sum_{i=1}^{d}||X_{i}f||_{\infty} + \cdots + \sum_{i_{1}, i_{2}, \ldots, i_{k} = 1}^{d}||X_{i_{1}}X_{i_{2}} \cdots X_{i_{k}}f||_{\infty}.$$

We say that $P$ is of {\it order $k$} if the same $k$ may be used in (\ref{dist1}) for all compact $K \subseteq G$.  We say that a distribution is {\it positive} if $f \in C_{c}^{\infty}(G)$ and $f \geq 0$ implies $Pf \geq 0$.

\begin{prop} \label{dist2}
Any positive distribution on $G$ is induced by a regular Borel measure $\mu$ in the sense that
$$ Pf = \int_{G}f d\mu,$$
for all $f \in C_{c}^{\infty}(G)$.
\end{prop}

 \begin{proof} This is proved in the same way as the corresponding result on Euclidean space. We first imitate the proof of Theorem 2.1.7 in \cite{Hor} to show that any positive distribution is of order zero, and then apply the Riesz representation theorem.
 \end{proof}

The {\it support} supp$(P)$ of a distribution $P$ is the set of all points in $G$ which have no open neighbourhood on which the restriction of $P$ vanishes. Let $U$ be an (open) canonical co--ordinate neighbourhood of $e$ (see e.g. Definition 2 in \cite{Chev} p.118) with co--ordinate functions $x_{1}, \ldots, x_{d}$, where we can, and will assume that $x_{i} \in C_{c}^{\infty}(G)$ for $i = 1, \ldots, d$. We denote by $\phi$, the mapping
$$ g \in U \rightarrow (x_{1}(g), \ldots, x_{d}(g)) \in \Rd.$$
Then $\phi$ is a homeomorphism from $U$ to an open neighbourhood $\widetilde{U}$ of $\Rd$ with $\phi(e) = 0$. For each $f \in C^{\infty}(U)$, we write $J_{\phi}f = f \circ \phi^{-1}$. So $J_{\phi}$ is a linear isomorphism between  $C^{\infty}(U)$ and $C^{\infty}(\widetilde{U})$. If we identify the vector field $X_{i}$ with its restriction to $U$, we have that there exists a vector field $\widetilde{X_{i}} = \sum_{j=1}^{d}c_{ij}(\cdot)\partial_{j}$ on $\widetilde{U}$, such that $\widetilde{X_{i}} = J_{\phi}X_{i}J_{\phi}^{-1}$ for all $i = 1, \ldots d$, where $c_{ij} \in C^{\infty}(\widetilde{U})$ with $c_{ij}(0) = \delta_{ij}$.

\begin{theorem} \label{dist3}
If $P$ is a distribution of order $k$ having support $\{e\}$, then it has the form
$$ Pf = \sum_{|\alpha| \leq k}c_{\alpha}X^{\alpha}f(e),$$
for all $f \in C_{c}^{\infty}(G)$, where $c_{\alpha} \in \R$ for each multi--index $\alpha$.
\end{theorem}

\begin{proof} If supp$(P) = \{e\}$, then $P$ fails to vanish on every co--ordinate neighbourhood of $e$. So in particular it fails to vanish on $U$. If $f \in C_{c}^{\infty}(U)$, we write $\widetilde{f} = J_{\phi}f \in C_{c}^{\infty}(\widetilde{U})$. We define a linear functional $\widetilde{P}: C_{c}^{\infty}(\widetilde{U}) \rightarrow \R$, by $\widetilde{P}\widetilde{f} = Pf$. We then have
\bean |\widetilde{P}\widetilde{f}| & = & |Pf| \\
& \leq & C \sum_{|\alpha| \leq k}||X^{\alpha}f||_{\infty, G}\\
& = & C \sum_{|\alpha| \leq k}||J_{\phi}X^{\alpha}f||_{\infty, \Rd}\\
&= & C \sum_{|\alpha| \leq k}||J_{\phi}J_{\phi}^{-1}{\widetilde X}^{\alpha}J_{\phi}f||_{\infty, \Rd}\\
& \leq & C^{\prime}\sum_{|\alpha| \leq k}||{\partial}^{\alpha}\widetilde{f}||_{\infty, \Rd}. \eean
From this we see that $\widetilde{P}$ may be extended to a distribution on $\Rd$ having support $\{0\}$. Then by Theorem 2.3.4 in \cite{Hor}, we have
$$ \widetilde{P}\widetilde{f} = \sum_{|\alpha| \leq k}d_{\alpha}\partial^{\alpha}\widetilde{f}(0),$$
for $d_{\alpha} \in \R, |\alpha| \leq k$. From this it follows that
$$ Pf = \sum_{|\alpha| \leq k}c_{\alpha}X^{\alpha}f(e),$$
as required.
\end{proof}

\section{The Positive Maximum Principle and a Generalised Courr\`{e}ge Theorem}

 A linear mapping $A: C_{c}^{\infty}(G) \rightarrow {\mathcal F}(G)$ satisfies the {\it positive maximum principle} (PMP) if $f \in C_{c}^{\infty}(G)$ and $f(\rho) = \sup_{\tau \in G}f(\tau) \geq 0 \Rightarrow Af(\rho) \leq 0$.

In this section we seek to give a global characterisation of all linear operators that satisfy the positive maximum principle (PMP). We will follow the approach developed on Euclidean space by Walter Hoh in his Habilitationschrift \cite{Hoh}. First we need to extend the PMP to real--valued linear functionals $T$ defined on the linear space $C_{c}^{\infty}(G)$. We say that $T$ satisfies the {\it positive maximum principle} (PMP), if whenever $f \in C_{c}^{\infty}(G)$ with $f(e) = \sup_{\tau \in G}f(\tau) \geq 0$ then $Tf \leq 0$. The linear functional $T$ is said to be {\it almost positive} if whenever $f \in C_{c}^{\infty}(G)$ with $f \geq 0$ and $f(e) = 0$ then $Tf \geq 0$. It is easily verified that if $T$ satisfies the PMP, then it is almost positive.

Next we explore the relationship between linear operators and associated linear functionals satisfying the positive maximum principle.
For any linear operator $A: C^\infty_c(G) \to {\mathcal F}(G)$, we define a family of linear functionals $A_g: C^\infty_c(G) \to \R$, $g\in G$ by
\begin{equation}\label{Ag}
A_g\varphi : = (L_gAL_{g^{-1}})\varphi(e) = A(L_{g^{-1}} \varphi)(g).
\end{equation}

\begin{lemma} \label{AgPMP}
	A linear operator $A:C^\infty_c(G) \to {\mathcal F}(G)$ satisfies the positive maximum principle if and only if for all $g\in G$ the functional $A_g$ satisfies the positive maximum principle.
\end{lemma}

\begin{proof}
	For necessity, let $\varphi\in C^\infty_c(G)$ be such that $\varphi(e) = \sup_{\sigma\in G} \varphi(\sigma)\geq 0$. Then for any $g\in G$, $L_{g^{-1}}\varphi\in C^\infty_c(G)$ with
	$ L_{g^{-1}}\varphi(g) = \sup_{\sigma\in G} L_{g^{-1}}\varphi(\sigma)\geq 0$. Now, since $A$ satisfies the PMP we get $A_g \varphi = AL_{g^{-1}}\varphi(g) \leq 0.$
	
	For sufficiency, let $\varphi\in C^\infty_c(G)$ be such that $\displaystyle \varphi(g)=\sup_{\sigma\in G}\varphi(\sigma)$ for some $g\in G$. Then $L_g\varphi$ satisfies $L_g\varphi(e) = \sup_{\sigma\in G} L_g\varphi(\sigma)$. The functional $A_g$ satisfies the PMP, thus we have
	$$
	A\varphi(g) = (L_g A L_{g^{-1}} L_g) \varphi(e) = A_g(L_g\varphi) \leq 0.
	$$
\end{proof}

The following result is an easy consequence of elementary calculus.

\begin{lemma}\label{localmax}
	Any function $\varphi\in C^\infty_c(G)$ with local maximum $\displaystyle \varphi(\sigma) = \sup_{g\in V} \varphi(g)$ for some $\sigma\in G$ on a neighbourhood $V$ of $\sigma$, satisfies $X\varphi(\sigma) = 0$ for all $X\in \g$.
\end{lemma}

The next result is crucial.

\begin{theorem} \label{order2}
	Let $T:C_c^\infty(G)\to \R$ be an almost positive functional, then $T$ is a distribution of order 2.
\end{theorem}

\begin{proof}
	In this proof we will use the fact that there is a canonical neighbourhood $V$ of $e$ such that $VV \subseteq U$ and $V^{-1} = V$, and so that $x_{i}(g^{-1}) = -x_{i}(g)$ for all $g \in V$ (see e.g. Proposition 2.1 (b) of \cite{Foll}, pp35--6). Let $\varphi\in C_c^\infty(G)$ with $H: = \text{supp} (\varphi)$. The set $H$ is covered by $\displaystyle \bigcup_{k\in H}l_kV$ and since it is compact, there is a finite subcover $\{ V_1, \dots V_N\}$, where $V_i:=l_{k_i}V$ for some $k_i\in H$, $i = 1,\dots , N$. Furthermore, since $G$ is a Hausdorff space, we can separate the points $k_1,k_2, \dots, k_N$ by a family of pairwise disjoint open sets $W_1, \dots, W_N \subset G$ such that $k_i\in W_i$ for all $i=1, \dots, N$. In the light of our construction of $V$, we may assume that all but possibly one of the $k_{1}, \ldots, k_{N} \notin V$ (indeed we may simply discard any that are). For the case where one of the $k_{i}$'s is in $V$, we will, without loss of generality, replace it with $e$.
	
Define $\widetilde{\varphi}\in C_c^\infty(G)$ by
	\begin{equation}\label{tilde_varphi}
	\widetilde{\varphi}(g) : = \varphi(g) - \sum_{l =1}^N \varphi(k_l)\varepsilon_l(g) - \sum_{l =1}^N \sum_{r =1}^d x_r(k_l^{-1}g)\; \varepsilon_l(g)\; X_r \varphi(k_l)
	\end{equation}
	for all $g \in G$, where for all $l=1,\dots, N$,$\varepsilon_l \in C_c^\infty(G)$ takes values in $[0,1]$ with support in $W_l$, and is equal to 1 in a neighbourhood of $k_l$.

It is straightforward to check that $\widetilde{\varphi}(k_l) = X_i\widetilde{\varphi}(k_m) = 0$, for all $l, m = 1, \ldots N, i=1,\dots d$. For any $g \in H$, there exists at least one $m = 1, \ldots, N$ such that $g \in V_{m}$ and for such a choice of $m$ we have

	$$\displaystyle \widetilde{\varphi}(g)= \widetilde{\varphi}(g) -\widetilde{\varphi}(k_m) -\sum_{i =1}^d x_i(k_m^{-1}g) \; X_i\widetilde{\varphi}(k_m).$$

Using this with a suitably translated form of Taylor's theorem on Lie groups (see e.g. \cite{Hel1}, p.105 or  \cite{App1}, pp.127--8) and the Cauchy--Schwarz inequality, we find that for all $g\in V_{m}$
	\bean
	|\widetilde{\varphi}(g)| & = &\left| \widetilde{\varphi}(g) -\widetilde{\varphi}(k_m) -\sum_{i =1}^d  x_i(k_m^{-1}g) \; X_i\widetilde{\varphi}(k_m)\right|\\
	& \leq & \frac{1}{2}\sum_{i,j =1}^{d} \left|x_i(k_m^{-1}g)x_j(k_m^{-1}g)\right| \; \sup_{v\in G}\left|X_iX_j\widetilde{\varphi}(v)\right|\\
	&\leq &  \frac{d}{2}D_{\varphi}\sum_{i=1}^{d} x_i(k_m^{-1}g)^2, \eean
where $D_{\varphi}:= \sup_{v\in G}\left|X_iX_j\widetilde{\varphi}(v)\right|$, and so for all $g \in V$,

$$ |\widetilde{\varphi}(k_{m}g)| \leq   \frac{d}{2} D_{\varphi}\sum_{i=1}^{d} x_i(g)^2.$$

We define a non--negative function $\psi \in C_{c}^{\infty}(G)$ such that for all $m = 1, \ldots, N,  V_{m} \subseteq$ supp($\psi$)  with $\psi(e) = 0$ and

$$ \psi(k_{m}g) \geq \sum_{i=1}^{d} x_i(g)^2$$

for all $g \in V$. Since $K \subseteq$ supp$(\psi)$ and $\widetilde{\varphi}(e)= \psi(e)=0$, by almost positivity of $T$ we deduce that
	\begin{equation*}
	|T(\widetilde{\varphi})| \leq \frac{d}{2} D_{\varphi}T(\psi).
	\end{equation*}
	The result then follows by the definition of $\widetilde{\varphi}$.
\end{proof}

A Borel measure $\mu$ on $G$ is a \textit{L\'evy measure} if $\mu(\{e\}) = 0$ and for every canonical co--ordinate neighbourhood $U$ of $e$ we have
	\begin{equation*}
	\int_{U} \left(\sum_{i =1}^d x_i^2(g)\right)  \mu(dg) < \infty, \; \text{and } \mu(U^c) < \infty,
	\end{equation*}
c.f. equation (1.8) in \cite{Liao} p.12.

We now establish a canonical form for linear functionals that satisfy the PMP. We closely follow the argument employed to prove Proposition 2.10 in \cite{Hoh} for the case $G = \R^{d}$ (see also Theorem 2.21 in \cite{BSW}, pp.47--50).

\begin{theorem} \label{PMPLF}
	Let $T:C_c^\infty(G)\to \R$ be a linear functional satisfying the positive maximum principle. Then there exists $c \geq 0, b = (b_{1}, \ldots, b_{d}) \in \Rd$, a non--negative definite symmetric $d \times d$ real--valued matrix $a = (a_{ij})$, and a L\'{e}vy measure $\mu$ on $G$ such that
	
	\begin{eqnarray} \label{distr_form}
	T\varphi & = & \sum_{i,j =1}^d a_{ij}X_iX_j \varphi(e) + \sum_{i=1}^d b_i X_i \varphi(e) - c\varphi(e)\nonumber \\ & + & \int_{G} \left(\varphi(g)-\varphi(e) - \sum_{i=1}^d x_i(g) X_i\varphi(e)\right) \mu(dg),
	\end{eqnarray}
	for all $\varphi\in C_c^\infty(G)$.
\end{theorem}

\begin{proof} We begin by remarking that $T$ is a distribution of order $2$ by Theorem \ref{order2}.
	Let $V$ be a compact neighbourhood of $e$ such that $\overline{U}\subseteq V$, and let $\phi_1, \phi_2\in C_c^\infty(G)$ be two functions taking values in $[0,1]$ with supports respectively in $V$ and in $U^C$, such that $\phi_1(g) = 1$ for all $g\in U$, and $\phi_2(g) = 1$ for all $g\in V^C$. We introduce the function $\displaystyle \xi := \sum_{i=1}^d x_i ^2 (\cdot) \cdot \phi_1 + \phi_2$ from $G$ to $\R$ and note that $\xi(e)=0$.
Then $\xi \cdot T$ is a positive distribution. Indeed, for all $\varphi \in C_c^\infty(G)$ such that $\varphi\geq 0$, we also have $\xi \cdot \varphi \geq 0$ with $\xi(e)\varphi(e) = 0$. So by almost positivity of $T$ we get
	\begin{equation*}
	\langle \xi \cdot T, \varphi \rangle = \langle  T, \xi \cdot \varphi \rangle \geq 0.
	\end{equation*}
	Since $\xi \cdot T$ is a positive distribution, by Proposition \ref{dist2} there exists a unique regular Borel measure $\nu$ on $G$ such that $\nu = \xi \cdot T $, i.e. $\langle \xi\cdot T , f \rangle =  \int_G f(g) \nu(dg)$ for all $f\in C^\infty_c(G)$.
	We define another Borel measure $\mu$ on $G$ by $\mu(\{e\}) = 0$ and $\mu: = \frac{1}{\xi} \cdot \left.\nu\right|_{G^*}$ on $G^*: = G \setminus \{e\}$. For all $f \in C^\infty_c(G)$ with $e \notin$ supp$(f)$ we have
	\begin{equation}\label{f_mu}
	\langle T, f\rangle = \int_{G}f(g)\mu(dg).
	\end{equation}
	We show that $\mu$ is a L\'{e}vy measure. By regularity of $\nu$ we have
	
	\begin{equation} \label{LevyCond1}
	\int_{U} \left(\sum_{i =1}^d x^2_i(g)\right)  \mu(dg) \leq \int_{V} \xi(g) \mu(dg) \leq \nu(V)< \infty
	\end{equation}
	
	Next let $\alpha, \beta \in C_c^\infty(G)$ be two functions taking values in $[0,1]$, such that $\text{supp} (\alpha)\subset U$ and $\text{supp} (\beta) \subset U^c$
	with $\alpha(e)=1$, so
	\begin{equation*}
	\alpha(e) + \beta(e) = 1 + 0 = \sup_{g\in G} (\alpha +\beta)(g).
	\end{equation*}
	By the positive maximum principle, $T(\alpha+\beta)\leq 0$, so by linearity of $T$, we get ${T\beta \leq - T \alpha}$, that is
	\begin{equation*}
	\la T, \beta \ra = \int_{G^*} \beta(g) \mu(dg) \leq - \la T, \alpha \ra
	\end{equation*}
	Taking the supremum over all possible $\beta$, it then follows that
	
	\begin{equation} \label{LevyCond2}
	\mu(U^{c}) \leq - \la T, \alpha \ra < \infty,
	\end{equation}
	so the measure $\mu$ is indeed a L\'evy measure.

Our next goal is to derive the form (\ref{distr_form}). To this end we introduce a linear functional $S: C_c^\infty(G) \to \R$, by
	\begin{equation*}
	S\varphi := \int_{G}\left[ \varphi(g) -\varphi(e)-\sum_{i=1}^d x_i(g)X_i \varphi(e)\right] \mu(dg), \qquad \text{for all } \varphi\in C_c^\infty(G).
	\end{equation*}
	The finiteness of the integral follows by Taylor's theorem, using (\ref{LevyCond1}) and (\ref{LevyCond2}). Using Lemma \ref{localmax}, for all $\varphi\in C^\infty_c(G)$ such that $\displaystyle \varphi(e) = \sup_{g\in G} \varphi(g) \geq 0$ we have $X_i \varphi(e) = 0$, therefore
	\begin{equation*}
	S\varphi = \int_{G}\left(\varphi(g)-\varphi(e) \right) \; \mu(dg)\leq 0.
	\end{equation*}
	That is, $S$ satisfies the positive maximum principle. By Theorem \ref{order2} both $T$ and $S$ are distributions of order 2, thus so is their difference $P := T-S$.
	If $\varphi\in C_c^\infty(G^{*})$, then
	\begin{equation*}
	S\varphi = \int_{G} \varphi(g) \mu(dg) = T\varphi.
	\end{equation*}
	That is $P\varphi = 0$ for all such $\varphi\in C^\infty_c(G^{*})$, therefore $\text{supp}(P)\subseteq \{e\}$. Thus, using Theorem \ref{dist3}, $P$ is of the form
	\begin{equation*}
	P\varphi = \sum_{i,j=1}^d a_{ij}X_iX_j \varphi(e) +\sum_{i=1}^d b_i X_i \varphi(e) - c\varphi(e).
	\end{equation*}
	We will now prove that the constant $c$ is positive. Let $(\varphi_k)_{k\in\N}$ be a sequence of non-negative, monotone increasing functions in $C_c^\infty(G)$ such that $\varphi_k =1$ in a neighbourhood of $e$, which is pointwise convergent to $\mathbbm{1}_G$, then by the monotone convergence theorem
	\begin{equation*}
	S \varphi_k = \int_{G} (\varphi_k(g) -1) \; \mu(dg) \rightarrow 0, \qquad \text{ as } k\to \infty.
	\end{equation*}
	So
	\begin{equation*}
	T\varphi_k = P\varphi_k + S\varphi_k = - c\varphi_k (e) + S\varphi_k \rightarrow -c, \qquad \text{ as } k\to \infty.
	\end{equation*}
	Note that for all $k\in \N$, $ \displaystyle \varphi_k(e) = \sup_{g\in G} \varphi_k(g)\geq 0$, and since $T$ satisfies the positive maximum principle we have $T\varphi_k \leq 0$. Thus, $c\geq 0$.
	
	It is clear that $(a_{ij})$ is symmetric. We will now prove that $(a_{ij})$ is also positive definite.
	Let $(\varepsilon_k)_{k\in \N}$ be a sequence of monotone decreasing functions in $C^\infty_c(G)$, taking values in $[0,1]$ such that $\varepsilon_k=1$ in a neighbourhood of $e$ with $V_k \subset V_{k'}$ when $k> k'$ and $\displaystyle \bigcap_{k\in \N} V_k = \{e\}$.
	Let us also denote by $f_\xi \in C^\infty_c(G)$ the function defined by $\displaystyle f_\xi(\cdot) := \frac{1}{2} \sum_{i,j=1}^d \xi_i \xi_j x_i(\cdot)x_j(\cdot)$, where $\xi = (\xi_1,\dots, \xi_d)\in \R^d$.
	Note that for all $1\leq k,l\leq d$, we have
	\begin{align*}
	X_lX_kf_\xi & = \frac{1}{2} \sum_{i,j=1}^d \xi_i \xi_j X_lX_k(x_ix_j)\\
	& = \frac{1}{2} \sum_{i,j=1}^d \xi_i \xi_j (\delta_{ki}\delta_{lj}+ \delta_{kj}\delta_{li})\\
	& = \xi_l \xi_k.
	\end{align*}
	Furthermore
	\begin{align*}
	T(\varepsilon_k \cdot f_\xi) & = P(\varepsilon_k \cdot f_\xi) + S(\varepsilon_k \cdot f_\xi) \\
	& = \sum_{i,j=1}^d a_{ij} \xi_i \xi_j +\int_{G} \varepsilon_k(g) f_\xi(g) \; \mu(dg).
	\end{align*}
	Thus, by dominated convergence $\displaystyle T(\varepsilon_k \cdot f ) \to \sum_{i,j=1}^d a_{ij} \xi_i \xi_j$ as $k$ goes to infinity.
	Also, since $T$ is almost positive, $T(\varepsilon_k \cdot f ) \geq 0$, so $ \displaystyle \sum_{i,j=1}^d a_{ij} \xi_i \xi_j \geq 0$ and $(a_{ij})$ is positive definite.
	
	It is clear that $(a_{ij})$ and $c$ are uniquely defined for all $T$. Moreover, $\nu$ was uniquely defined and therefore $\mu$ is also uniquely defined. This allows as to calculate $\displaystyle \sum_{i=1}^d b_i X_i\varphi(e)$ for any $\varphi\in C^\infty_c(G)$, so the vector $b = (b_1, \dots, b_d)\in \R^d$ is also uniquely defined.
\end{proof}

Next we establish the converse to Theorem \ref{PMPLF}.

\begin{theorem}
	Every linear functional on $C^\infty_c(G)$ of the form (\ref{distr_form}) satisfies the positive maximum principle.
\end{theorem}

\begin{proof}
	By Lemma \ref{localmax}, for any function $\varphi\in C_c^\infty(G)$ such that $\displaystyle \varphi(e) =\sup_{g\in G}  \varphi(g)\geq 0$, we have $X\varphi(e) =0$ for all $X\in \g$.
	Thus, in (\ref{distr_form}) both the first order differential part and the integral part satisfy the positive maximum principle. So now we only have to deal with the second order differential part.\\
	Fix $h\in U$ and let us consider the function $\displaystyle H : t \mapsto \varphi\left(\exp\left(t\sum_{i =1}^d x_i(h)X_i\right)\right)$ from $\R$ to $G$. Let $P_2$ be the second order Taylor polynomial of $H$ around $0$,
	$$
	P_2(t) = \varphi(e) + t \sum_{i =1}^d x_i(h) X_i\varphi(e) + \frac{t^2}{2} \sum_{i, j=1}^d x_i(h) x_j(h) X_iX_j \varphi(e)
	$$
	Then $H(t) -P_2(t) = o(t^2)$ as $t\to 0$. That is
	\begin{equation*}
	\varphi\left(\exp\left(t\sum_{i =1}^d x_i(h)X_i\right)\right) - \varphi(e) - t \sum_{i =1}^d x_i(h) X_i\varphi(e) - \frac{t^2}{2} \sum_{i, j=1}^d x_i(h) x_j(h) X_iX_j \varphi(e) =o(t^2)
	\end{equation*}
	Thus, by Lemma \ref{localmax} and the fact that $\displaystyle \varphi(e) = \sup_{g\in G} \varphi(g)$ we have
	\begin{equation*}
	o(t^2) + \frac{t^2}{2} \sum_{i, j=1}^d x_i(h) x_j(h) X_iX_j \varphi(e) = \varphi\left(\exp\left(t\sum_{i =1}^d x_i(h)X_i\right)\right) - \varphi(e) \leq 0
	\end{equation*}
	Hence, dividing by $t^2$ and taking the limit $t\to 0$ we get
	\begin{equation*}
	\sum_{i, j=1}^d x_i(h) x_j(h) X_iX_j \varphi(e) \leq 0
	\end{equation*}
	
	The map $h\mapsto (x_1(h), \dots, x_d(h))$ is a diffeomorphism from $U$ to an open neighbourhood $\widetilde{U}$ of $0\in \R^d$.
	Let us fix an open ball $B_r(0)\subset \widetilde{U}$ of radius $r>0$. Then given any $\lambda\in \R^d$, there exists $c(\lambda)>0$ such that $c(\lambda) \lambda\in B_r(0)$, i.e. for all $i=1,\dots, d$, $c(\lambda) \lambda_i= x_i(h)$ for some $h\in U$. Then we have
	\begin{equation*}
	\sum_{i, j=1}^d \lambda_i \lambda_j X_iX_j f(e) = \frac{1}{c(\lambda)^2}\sum_{i, j=1}^d x_i(h)x_j(h)X_iX_jf(e) \leq 0
	\end{equation*}

We conclude that any linear functional on $C^\infty_c(G)$ of the form (\ref{distr_form}) satisfies the positive maximum principle.

\end{proof}

We may now proceed to the main result of this section, namely the generalisation of Courr\`ege's theorem from Euclidean space to Lie groups. We will need the concept of a {\it L\'{e}vy kernel} on $G$ which is a mapping $\mu:G \times {\mathcal B}(G) \rightarrow [0, \infty]$ such that $\mu(g, \cdot)$ is a L\'{e}vy measure on $G$ for all $g \in G$.

\begin{theorem} \label{PMP1} The mapping $A: C_{c}^{\infty}(G) \rightarrow {\mathcal F}(G)$ satisfies the PMP if and only if there exist functions $c, b_{i}, a_{jk}~(1 \leq i,j,k \leq d)$ from $G$ to $\R$, wherein $c$ is non--negative, and the matrix $a(\sigma): = (a_{jk}(\sigma))$ is non--negative definite and symmetric for all $\sigma \in G$, and a L\'{e}vy kernel $\mu$, such that for all $f \in C_{c}^{\infty}(G), \sigma \in G$,

\begin{eqnarray} \label{PMP2}
Af(\sigma) & = & -c(\sigma)f(\sigma) + \sum_{i=1}^{d}b_{i}(\sigma)X_{i}f(\sigma) + \sum_{j,k = 1}^{d}a_{jk}(\sigma)X_{j}X_{k}f(\sigma) \nonumber \\
& + & \int_{G}\left(f(\sigma \tau) - f(\sigma) - \sum_{i=1}^{d}x_{i}(\tau)X_{i}f(\sigma)\right)\mu(\sigma, d\tau).
\end{eqnarray}

\end{theorem}

\begin{proof}
	From Lemma \ref{AgPMP}, we know that $A$ satisfies the PMP if and only if for all $g\in G$ the linear functional $A_g$ as defined in (\ref{Ag}) satisfies the PMP.
	Thus, for each $g\in G $, the functional $A_g$ has the form (\ref{distr_form}),
	\bean
	A_g \varphi & = & \sum_{i,j =1}^d a_{ij}(g) X_i X_j \varphi(e) - \sum_{i=1}^d b_i(g)X_i \varphi(e) - c(g) \varphi(e)  \\
	& + & \int_{G} \left( \varphi(\tau) - \varphi(e) \sum_{i,j =1}^d x_i(\tau) X_i \varphi(e) \right) \mu(g, d\tau).
	\eean
	
	Then in particular for the function $L_g\varphi \in C^\infty_c(G)$, we get
	\bean
	A\varphi(g) = A_g(L_g\varphi) = & \sum_{i,j =1}^d a_{ij}(g) X_i X_j \varphi(g) - \sum_{i=1}^d b_i(g)X_i \varphi(g) - c(g) \varphi(g)  \\
	& + \int_{G} \left( \varphi(g\tau) - \varphi(g) - \sum_{i =1}^d x_i(\tau) X_i \varphi(g) \right) \mu(g, d\tau),
	\eean
and the result is established.	
\end{proof}

In the last part of this section, we will establish sufficient conditions for $A: C_{c}^{\infty}(G) \rightarrow C_{0}(G)$. We introduce some simplifying notation.  Define $H: C_{c}^{\infty}(G) \rightarrow C(G \times G)$ by
	\begin{equation}\label{H}
	Hf(g,\tau):= f(g\tau) - f(g) - \sum_{i=1}^d x_i(\tau) X_if(g),
	\end{equation}
for all $ f, \in C_{c}^{\infty}(G), g, \tau \in G$. We then write $A = A_{1} + A_{2} + A_{3}$ in (\ref{PMP2}), where $A_{2}f(g) = \int_{U}Hf(g,\tau)\mu(g, d \tau)$, and $A_{3}f(g) = \int_{U^{c}}Hf(g,\tau)\mu(g, d \tau)$.

If $X \in {\mathcal B}(G)$, then $C_{b}(X)$ denote the space of all real--valued bounded continuous functions defined on $X$.

\begin{theorem} \label{contA}
	If $A: C_{c}^{\infty}(G) \rightarrow {\mathcal F}(G)$ satisfies the positive maximum principle such that in (\ref{PMP2})
	\begin{itemize}
		\item[i)] The functions $a_{ij}, b_k, c$  are continuous for all $i,j,k=1,\dots, d$
		\item[ii)] The mappings $g \rightarrow \int_{U}h(\tau)\left(\sum_{i=1}^{d}x_{i}(\tau)^{2}\right)\nu(g, d\tau)$ are continuous from $G$ to $[0, \infty)$, for all $h \in C_{b}(U)$, and the mappings $g \rightarrow \int_{U^{c}}k(\tau)\mu(g, d\tau)$ are continuous from $G$ to $[0, \infty)$, for all $k \in C_{b}(U^{c})$,
		\end{itemize}
	then Ran$(A) \subseteq C(G)$.
\end{theorem}

\begin{proof} It is clear that if (i) holds,  the differential operator term $A_{1}$ in (\ref{PMP2}) has continuous range, and so we concentrate on $A_{3}$ and $A_{2}$ (in that order). Let $(g_{n})$ be a sequence in $G$ converging to $g \in G$ as $n \rightarrow \infty$. If $k: = {\bf 1}_{U^{c}}$ then $k \in C_{b}(U^{c})$, and by assumption (ii) we have that
$$\sup_{n \in N}\mu(g_{n}, U^{c}) = \sup_{n \in N}\int_{U^{c}}k(\tau)\mu(g_{n}, d\tau) < \infty.$$ A similar argument shows that $\sup_{\nN}\int_{U}\sum_{i=1}^{d}x_{i}(\tau)^{2}\mu(g_{n}, d\tau) < \infty$.

 For each $\nN$, we have
	\bean & & |A_{3}f(g) - A_{3}f(g_{n})|\\
	&  \leq & \left|\int_{U^{c}}(Hf(g, \tau) - Hf(g_{n}, \tau))\mu(g_{n}, d\tau)\right| \\
	& + & \left|\int_{U^c} Hf(g,\tau)\mu(g, d\tau) - \int_{U^c} Hf(g,\tau)\mu(g_{n}, d\tau)\right|
	\eean

The first term is majorized by $\sup_{\tau \in U^{c}}|H(g, \tau) - H(g_{n}, \tau)|\sup_{n \in N}\mu(g_{n}, U^{c}) \rightarrow 0$ as $n \rightarrow \infty$. The second term converges to zero as $n \rightarrow \infty$ by assumption (ii).

	We will now deal with the operator $A_{2}$. For all $g\in G$, we have
	\bean
	& & |A_{2}f(g) - A_{2}f(g_{n})|\\
	& \leq & \left|\int_{U} \left[ Hf(g,\tau)-Hf(g_n,\tau)\right] \;\mu(g_{n}, d\tau)\right|\\ & + & \left|\int_{U} Hf(g,\tau)\mu(g, d\tau) - \int_{U} Hf(g,\tau)\mu(g_{n}, d\tau)\right|.
	\eean
	By Taylor's formula on Lie groups, for all $\tau \in U$ there exist $\tau', \tau''\in U$ such that
	\begin{equation*}
	Hf(g,\tau) -  Hf(g_n,\tau) = \frac{1}{2} \sum_{i, j=1}^d x_i(\tau)x_j(\tau)(X_iX_jf(g\tau') -X_iX_jf(g_n\tau''))
	\end{equation*}
Using the Cauchy--Schwarz inequality, we see that $\frac{Hf(g,\tau) -  Hf(g_n,\tau)}{\sum_{i=1}^{d} x_{i}(\tau)^{2}}$ is bounded (and continuous) in $\tau \in U$, hence the first term goes to zero as $n \rightarrow \infty$, by continuity and assumption (ii). To verify that the second term also goes to zero as $n \rightarrow \infty$, we again use assumption (ii), together with the fact that (again by Taylor's theorem) $\frac{H(g, \tau)}{\sum_{i=1}^{d} x_{i}(\tau)^{2}}$ is bounded (and continuous) in $\tau \in U$.

\end{proof}

\begin{theorem} \label{vanishinf}
If $A: C_{c}^{\infty}(G) \rightarrow {\mathcal F}(G)$ satisfies the positive maximum principle such that for the L\'{e}vy kernel $\mu$ in (\ref{PMP2}) we have
\begin{itemize}
		\item[iii)]
		$$
		\lim_{\sigma\to \infty} \mu(\sigma, U^{c}) = 0,
		$$
	\end{itemize}
	 then $Af$ vanishes at infinity for all $f\in C^\infty_c(G)$.
\end{theorem}

\begin{proof} Since differential operators cannot increase supports, it is sufficient to consider the operator $A_{2} + A_{3}$.
We can and will assume that $U$ has compact closure $\overline{U}$. Let $f \in C_{c}^{\infty}(G)$ with $C :=$supp$(f)$. Since $B_{1}:= C\overline{U^{-1}}$ is compact and $C \subseteq B_{1}$, we have
$Hf(\sigma,\tau) = 0$ for all $\sigma \in B_{1}^{c}, \tau \in U$, and hence $A_{2}f(\sigma) = 0$ for all $\sigma \in B_{1}^{c}$.

	We easily compute that there exists $K_{f} > 0$ so that for all $\sigma \in G$,
	\begin{equation*}
	|A_{3}f(\sigma)| \leq K_{f} \cdot \mu(\sigma, U^c).
	\end{equation*}
	Using condition (iii), there is a compact set $B_{2}$ such that for all $\sigma\in B_{2}^c$,
	$$
	 |A_{3}f(\sigma)| < \varepsilon
	$$
	Since $B_{1} \cup B_{2}$ is compact, we get
	$$
	|A_{2}f(\sigma) + A_{3}f(\sigma)| < \varepsilon \qquad \text{for all } \sigma \in (B_{1} \cup B_{2})^c,
	$$
	and the proof is complete.

\end{proof}

Combining the results of the last two theorems, we obtain

\begin{cor} \label{nice}
If $A: C_{c}^{\infty}(G) \rightarrow {\mathcal F}(G)$ satisfies the positive maximum principle and conditions (i) and (ii) from Theorem \ref{contA} and condition (iii) from Theorem \ref{vanishinf} are satisfied, then $A$ maps $C^\infty_c(G)$ to $C_0(G)$.
\end{cor}

\vspace{5pt}

{\bf Remark}. If we work in a local co--ordinate system, then after some routine manipulations, we can rewrite (\ref{PMP2}) in the form obtained in \cite{BCP}.


\section{The Generalised Courr\`{e}ge Theorem, Feller Semigroups and a Killed Hunt's formula}

Let $(\Omega, \mathcal{F}, \mathbb{P})$ be a probability space where the $\sigma$-algebra $\mathcal{F}$ is equipped with a filtration $(\mathcal{F}_t)_{t\geq 0}$. We will consider a Markov process $(Y_t)_{t\geq 0}$ with respect to the filtration $(\mathcal{F}_t)_{t\geq 0}$, taking values in $G$. We will denote the transition probabilities of $(Y_t)_{t\geq 0}$ by ${p_t(\sigma, A):= P(Y_t\in A| Y_0= \sigma)}$ for all $A\in \mathcal{B}(G)$ and $\sigma\in G$. We then have a one-parameter contraction semigroup of operators on $B_b(G)$ given for each $f \in B_{b}(G), \sigma \in G$ by
\begin{equation}\label{1.Tf}
T_t f (\sigma) = \int_G f(\tau) p_t(\sigma, d\tau).
\end{equation}
The family of operators $(T_t)_{t\geq 0}$ is called a \textit{Feller semigroup} if
\begin{itemize}
	\item[1)] $T_tC_0(G)) \subseteq C_0(G)$ for all $t\geq 0$,
	\item[ii)] $\displaystyle \lim_{t\to 0} \|T_tf -f \|_\infty = 0$ for all $f\in C_0(G)$.
\end{itemize}
In this case we say that $(Y_t)_{t\geq 0}$ is a Feller process. We will denote by $A$ the infinitesimal generator of  the Feller semigroup $(T_t)_{t\geq 0}$. The following lemma is well-known, see \cite{Jac1}, p.332. We will include the proof for completeness.
\begin{lemma}\label{H.Generator_PMP}
	If $(T_t)_{t\geq 0}$ is a Feller semigroup with generator $A$ such that $C_c^\infty(G) \subseteq \text{Dom}(A)$, then $A$ satisfies the positive maximum principle.
\end{lemma}

\begin{proof}
	Let $f\in C_c^\infty(G)$ be such that $\displaystyle f(g_0) = \sup_{\sigma\in G} f(\sigma) \geq 0$ for some $g_0\in G$. We have
	\begin{equation*}
	Af(g_0) = \lim_{t\to 0} \frac{(T_t f - f)(g_0)}{t}  =\lim_{t\to 0} \frac{1}{t} \int_G \left[f(\tau) - f(g_0) \right]p_t(g_0, d\tau) \leq 0.
	\end{equation*}
	This proves that A satisfies the PMP.
\end{proof}

Since $A$ satisfies the PMP, we may associate a unique L\'{e}vy kernel to it by Theorem \ref{PMP1}. We will now explore the relationship between the transition probabilities of the Feller process and this associated L\'evy kernel. This generalises a well-known result about convolution semigroups in $\R^d$ (see \cite{Sa}, Corollary 8.9, p.45), which has recently been extended to the current context, in the case where $G = \R^{d}$, in Theorem 3.2 of \cite{KS}.
\begin{prop}
	Let $(p_t)_{t\geq 0}$ be the transition probabilities associated to a Feller process, with Feller semigroup $(T_t)_{t\geq 0}$ and generator $A$ which has associated L\'{e}vy kernel $\mu$. Assume that $C^\infty_c(G)\subseteq \text{Dom}(A)$, then
	\begin{equation*}
	\lim_{t \to 0}\frac{1}{t} \int_G f(\tau) p_t(g, d\tau) = \int_{G} f(g\tau) \mu(g, d\tau)
	\end{equation*}
	for all $g\in G$ and $f\in C^\infty_c(G)$ vanishing on a neighbourhood of $g$, where $\mu$ is the L\'evy kernel associated to $A$.
\end{prop}
\begin{proof}
	By definition, we have for all $f\in C^\infty_c(G)$, $g\in G$
	\begin{equation}\label{eq57}
	Af(g) = \lim_{t \to 0}\frac{1}{t} \int_G (f(\tau)-f(g)) p_t(g, d\tau)
	\end{equation}
	Given that the generator $A$ satisfies the PMP, by Lemma \ref{AgPMP} for all $g\in G$ the distribution $A_g$ also satisfies the PMP, where $A_g$ is defined in (\ref{Ag}).
	Using (\ref{f_mu}) we have that for all $g\in G$ and $f\in C^\infty_c(G)$, $A_gf = \int_{G} f(\tau) \;\mu(g, d\tau) $ where $\mu(g, \cdot)$ is a L\'evy measure.
	Thus, for all $g\in G$ and $f\in C^\infty_c(G)$ we have
	\begin{equation}\label{eq58}
	Af(g) = A_g(L_gf) = \int_{G} f(g\tau) \;\mu(g, d\tau)
	\end{equation}
	In particular when $f\in C^\infty_c(G)$ vanishes on a neighbourhood of $g\in G$, from (\ref{eq57}) and $(\ref{eq58})$ we get
	\begin{equation*}
	\lim_{t \to 0}\frac{1}{t} \int_G f(\tau) p_t(g, d\tau)= \int_{G} f(g\tau) \;\mu(g, d\tau).
	\end{equation*}
\end{proof}

A finite Borel measure on $G$ is a sub--probability measure if its total mass does not exceed one, and the family $(\rho_t)_{t\geq 0}$ of sub--probability measures  on $(G, \mathcal{B}(G))$, is called a \textit{convolution semigroup of sub--probability measures} if
	\begin{itemize}
		\item[i)]$\rho_{s+t} = \rho_s \ast \rho_t, \qquad \text{for all } s,t \geq 0$
		\item[ii)] $\rho_0 = \delta_e$,
		\item[iii)] $\displaystyle \lim_{t \to 0} \rho_t = \delta_e$, in the sense of weak convergence,
	\end{itemize}
where $*$ denotes the usual convolution of measures (see e.g. section 4.1 of \cite{App1} pp.82--3).   We have a convolution semigroup of probability measures if $\rho_{t}(G) = 1$ for all $t \geq 0$. The latter arise as the laws of L\'evy processes on $G$, i.e. processes with stationary and independent increments that are stochastically continuous, see e.g. p.10 in \cite{Liao}. A L\'evy process is a Feller process and the contraction semigroup, $(T_t)_{t\geq 0}$ defined on $C_0(G)$, is called a \textit{Hunt  semigroup} and is defined for all $t \geq 0, f \in C_{0}(G), \sigma \in G$ by
\begin{equation} \label{Husemi}
T_t f(\sigma) = \int_G f(\sigma \tau) \rho_t(d\tau),
\end{equation}
If we compare this to (\ref{1.Tf}), we see that
$$p_t(\sigma, A) =\rho_t(\sigma^{-1}A)$$
for all $\sigma\in G$ and $A\in \mathcal{B}(G)$.
The infinitesimal generator ${\mathcal L}$ of a Hunt semigroup is called the \textit{Hunt generator}.
From the definition of the Hunt semigroup it follows that $L_\sigma T_t = T_t L_\sigma$ for all $\sigma \in G$ and $t\geq 0$. Therefore, since by Lemma 5.3.2 in \cite{App1}, for all $f\in \text{Dom}({\mathcal L}), \sigma \in G$, $L_\sigma f \in \text{Dom}({\mathcal L})$  then  $L_\sigma {\mathcal L} f = {\mathcal L} L_\sigma f $.

The Hunt semigroups are precisely the (left) translation invariant Feller semigroups, as the next result shows.
\begin{prop} \label{Hu2}
If $(T_{t}, t \geq 0)$ is a Feller semigroup in $C_{0}(G)$ such that $p_{0}(e, \cdot) = \delta_{e}(\cdot)$ and $L_{\sigma}T_{t} = T_{t}L_{\sigma}$ for all $t \geq 0, \sigma \in G$, then $(T_{t}, t \geq 0)$ is the Hunt semigroup associated with a convolution semigroup of probability measures.
\end{prop}

This is proved exactly as in the Euclidean case (see e.g. Theorem 3.3.1 in \cite{Appbk} pp.161--2).

We have the following celebrated classification of the Hunt generator, originally due to Hunt \cite{Hu}. First we define the space
$$ C_{0}^{(2)}(G):= \{f \in C_{0}(G); X_{i}f \in C_{0}(G)~\mbox{and}~X_{j}X_{k}f \in C_{0}(G)~\mbox{for all}~1 \leq i,j,k \leq d\}.$$
Clearly $C_{c}^{\infty}(G) \subseteq C_{0}^{(2)}(G)$, and so $C_{0}^{(2)}(G)$ is dense in $C_{0}(G)$.

\begin{theorem}[Hunt's theorem] \label{Hunt}
If $(\rho_{t}, t \geq 0)$ is a convolution semigroup of probability measures
in $G$ with generator ${\mathcal L}$, then
\begin{enumerate}
\item $C_{0}^{(2)}(G) \subseteq \mbox{Dom}({\mathcal L})$. \item For each
$\sigma \in G, f \in C_{0}^{(2)}(G)$,
\begin{eqnarray} \label{hu}
{\mathcal L} f(\sigma) & = & \sum_{i=1}^{d}b_{i}X_{i}f(\sigma) +
\sum_{i,j=1}^{d}a_{ij}X_{i}X_{j}f(\sigma)\nonumber \\
 & + & \int_{G}\left(f(\sigma \tau) - f(\sigma) -
   \sum_{i=1}^{d}x_{i}(\tau)X_{i}f(\sigma)\right)\nu(d\tau), \nonumber \\
   & &
\end{eqnarray}

where $b = (b_{1}, \ldots, b_{d}) \in {\R}^{d}, a = (a_{ij})$ is a
non-negative definite, symmetric $d \times d$ real-valued matrix
and $\nu$ is a L\'{e}vy measure on $G$.
\end{enumerate}
Conversely, any linear operator with a representation as in
(\ref{hu}) is the restriction to $C_{0}^{(2)}(G)$ of the Hunt generator\index{Hunt generator}
corresponding to a unique convolution semigroup of probability
measures.
\end{theorem}

For a proof of this result see \cite{Hu}, section 4.2 of \cite{He1} pp.259--69 or section 3.1 of \cite{Liao} pp.52--61.

 Let $c > 0$ and $(\rho_{t}, t \geq 0)$ be a convolution semigroup of probability measures in $G$,  and consider the family of measures $(\widetilde{\rho}_{t}, t \geq 0)$ where for each $t \geq 0, \widetilde{\rho}_{t} = e^{-ct}\rho_{t}$. Then $(\widetilde{\rho}_{t}, t \geq 0)$ is a convolution semigroup of sub--probability measures, since for each $t > 0, \widetilde{\rho}_{t}(G) = e^{-ct} < 1$. We obtain a $C_{0}$--contraction semigroup $(S_{t}, t \geq 0)$ on $C_{0}(G)$ given by $S_{t} = e^{-ct}T_{t}$ for each $t \geq 0$, and we have $S_{t}f(\sigma) = \int_{G}f(\sigma \tau)\widetilde{\rho}_{t}(d\tau)$ for all $f \in C_{0}(G), \sigma \in G$. Clearly the action of the infinitesimal generator ${\mathcal M}$ of this semigroup is given by
${\mathcal M}f = -cf + {\mathcal L}f$, for all $f \in C_{0}^{(2)}(G)$. Note that $(S_{t}, t \geq 0)$ is not a Feller semigroup in our sense. It can be naturally associated with a {\it killed L\'{e}vy process} whose state space is the one--point compactification of $G$, and $c$ is then interpreted as a killing rate. For details see e.g. \cite{Appbk} pp.405--6. We thus call ${\mathcal M}$ a {\it killed Hunt generator}. It is clear, e.g. by the argument of the proof of Lemma \ref{H.Generator_PMP} that it satisfies the PMP. In the converse direction, we have the following result.

\begin{theorem} \label{Hu1}
If $A:C_{c}^{\infty}(G) \rightarrow C(G)$ satisfies the positive maximum principle and is such that $L_{g}Af = AL_{g}f$ for all $f \in C^\infty_c(G), g \in G$ then
\begin{eqnarray} \label{Huntf}
Af(g) & = & \sum_{i,j =1}^d a_{ij}X_iX_j f(g) + \sum_{i=1}^d b_i X_if(g) - cf(g) \nonumber \\
& + &\int_{G} \left(f(g\tau)-f(g) - \sum_{i=1}^d x_i(\tau) X_if(g)\right) \mu(d\tau),
\end{eqnarray}
where $\{a_{ij}\}$ is a non-negative definite, symmetric matrix, ${(b_1, \dots, b_d)\in \R^d}, c \geq 0$ and $\mu$ is a L\'evy measure on $(G, \mathcal{B}(G))$. Furthermore $A$ extends to the killed Hunt generator associated to a unique convolution semigroup of sub--probability measures on $G$.
\end{theorem}

\begin{proof}
	Since $A$ satisfies the PMP, by Theorem \ref{PMP1} it is of the from (\ref{PMP2}). Furthermore, $A$ is invariant under left translation on $C^\infty_c(G)$. Thus,
	\begin{align*}
	Af(g) & = AL_gf(e) = L_g Af(e) \\
	& = \sum_{i,j =1}^d a_{ij}(e)X_i X_j f(g) + \sum_{i=1}^d b_i(e)X_if(g) -c(e)f(g) \\
	& + \int_{G}\left[f(g\tau) - f(g) - \sum_{i=1}^d x_i(\tau)X_if(g) \right]\mu(e, d\tau ).
	\end{align*}
The form (\ref{Huntf}) follows when we define $a_{ij} = a_{ij}(e)$, $b_k = b_k(e), c =c(e)$ for all $i,j,k=1, \dots, d$, and $\mu(\cdot) = \mu(e, \cdot)$.
Next write (\ref{Huntf}) as $Af = -cf + Bf$. It is clear that $A$ and $B$ may be extended to linear operators $A_{1}$ and $B_{1}$ on $C_{0}^{(2)}(G)$, so that $A_{1}f = -cf + B_{1}f$ for all $f \in C_{0}^{(2)}(G)$. Then by Hunt's theorem, $B_{1}$ extends to the Hunt generator associated to a unique convolution semigroup $(\rho_{t}, t \geq 0)$ of probability measures, and then the required convolution semigroup of sub--probability measures is given, as above, by defining $\widetilde{\rho}_{t} = e^{-ct}\rho_{t}$ for each $t \geq 0$. Moreover by Theorem 5.3.4 on pp.137 of \cite{Appbk}, $C_{c}^{\infty}(G)$ is a core for the Hunt generator, and so for the killed Hunt generator, from which we see that the action of $A$ on $C_{c}^{\infty}(G)$ uniquely determines $(\widetilde{\rho}_{t}, t \geq 0)$.
\end{proof}

\section{Pseudo--Differential Operator Representation}

In this section, we assume that the conditions of Corollary \ref{nice} hold so that $A$ maps $C_{c}^{\infty}(G)$ to $C_{0}(G)$. Since any operator in $C_{0}(G)$ that satisfies the positive maximum principle is closeable (see e.g. \cite{EK} Lemma 2.11, p.16), 
$A$ has a closed extension $\overline{A}$ with Dom$(\overline{A}) \supseteq C_{c}^{\infty}(G)$.

We will also assume that $G$ is compact. It will be equipped with normalised bi--invariant Haar measure. Let $\G$ be the {\it unitary dual} of $G$, i.e. the set of all (equivalence classes with respect to unitary conjugation) of irreducible unitary representations of $G$. If $\pi \in \G$ then $\pi$ acts as unitary matrices on a complex Hilbert space $V_{\pi}$ having dimension $d_{\pi} < \infty$.  We denote the associated {\it derived representation} of $\g$, acting as skew--Hermitian matrices on $V_{\pi}$ by $d\pi$. Choosing bases, once and for all, in $V_{\pi}$, for each $\pi \in\ G$, we may define the co--ordinate functions in the usual way as $\pi_{ij}:G \rightarrow \C$, given by $\pi_{ij}(g):= \pi(g)_{ij}$ for each $1 \leq i, j \leq d_{\pi}$. These mappings are known to be $C^{\infty}$. For the purposes of this article, we will say that a linear operator $T:C^{\infty}(G) \rightarrow C(G)$ is a {\it pseudo--differential operator} if there is a mapping $j_{T}:G \times \G \rightarrow \bigcup_{\pi \in \G}M_{d_{\pi}}(\C)$ so that $j_{T}(\sigma, \pi) \in M_{d_{\pi}}(\C)$ for all $\sigma \in G, \pi \in \G$ and that for all $f \in C^{\infty}(G), \sigma \in G$,
$$ Tf(\sigma) = \sum_{\pi \in \G}d_{\pi}\tr(j_{T}(\sigma, \pi)\widehat{f}(\pi)\pi(\sigma)),$$
where the matrix $\widehat{f}(\pi) = \int_{G}f(\tau^{-1})\pi(\tau)d\tau$ is the Fourier transform of $f$. We then say that the mapping $j_{T}$ is the {\it symbol} of the operator $T$. For background on operators of this type, see \cite{RT, App2}. Our goal is now to prove that if $A:C^{\infty}(G) \rightarrow C_{0}(G)$ satisfies the PMP, then it is a pseudo--differential operator. First we find the candidate to be the symbol of such an operator.

Taking $f = \pi_{ij}$ in (\ref{PMP2}) where $\pi \in \G$, we may for each $\sigma \in G$, consider the matrix $A\pi(\sigma)$ whose $(i, j)$th entry is $A\pi_{ij}(\sigma)$, for $1 \leq i, j \leq d_{\pi}$. Define the matrix valued function $$j_{A}(\sigma, \pi): = \pi(\sigma)^{-1}A\pi(\sigma).$$ In the following, for each $\pi \in \G, I_{\pi}$ denotes the identity matrix acting in $V_{\pi}$.

\begin{prop} \label{PDo1}
For each $\pi \in \G, \sigma \in G$,
\begin{eqnarray} \label{PDo2}
j_{A}(\sigma, \pi) & = & -c(\sigma)I_{\pi} + \sum_{i=1}^{d}b_{i}(\sigma)d\pi(X_{i}) + \sum_{j,k = 1}^{d}a_{jk}(\sigma)d\pi(X_{j})d\pi(X_{k}) \nonumber \\
& + & \int_{G}\left(\pi(\tau) - I_{\pi} - \sum_{i=1}^{d}x_{i}(\tau)d\pi(X_{i})\right)\mu(\sigma, d\tau).
\end{eqnarray}

\end{prop}

\begin{proof} This follows in a straightforward manner from (\ref{PMP2}) and the definition of $j_{A}$, using standard properties of representations, and the fact that (with an obvious notation),
$$ X \pi(\sigma) = \left.\frac{d}{du}\pi(\sigma \exp(u X))\right|_{u=0} = \pi(\sigma) d\pi(X),$$
for all $X  \in \g$.
\end{proof}

We want to show that $A$ is a pseudo--differential operator with symbol $j_{A}$. Define ${\mathcal E}(G)$ to be the linear span of $S(G):= \{\sqrt{d_{\pi}}\pi_{ij}; 1 \leq i, j \leq d_{\pi}, \pi \in \G\}$. Then by the Peter--Weyl theorem, $S(G)$ is a complete orthonormal basis for $L^{2}(G)$, and ${\mathcal E}(G)$ is norm dense in $L^{2}(G)$, and uniformly dense in $C(G)$. For all $f \in {\mathcal E}(G), \sigma \in G$ we have the Fourier expansion
\begin{equation} \label{FourT}
  f(\sigma) = \sum_{\pi \in \G}d_{\pi}\tr(\widehat{f}(\pi)\pi(\sigma)),
\end{equation}
and the right hand side of (\ref{FourT}) is in fact a finite sum.

Using the definition of $j_{A}$, we obtain

\begin{eqnarray} \label{PDo3}
 Af(\sigma) & = &  \sum_{\pi \in \G}d_{\pi}\tr(\widehat{f}(\pi)A\pi(\sigma)) \nonumber \\
& = &   \sum_{\pi \in \G}d_{\pi}\tr(\widehat{f}(\pi)\pi(\sigma)j_{A}(\sigma, \pi)) \nonumber \\
& = &   \sum_{\pi \in \G}d_{\pi}\tr(j_{A}(\sigma, \pi)\widehat{f}(\pi)\pi(\sigma)).
\end{eqnarray}

We need to show that the right hand side of (\ref{PDo3}) converges to $Af(\sigma)$ for all $f \in C^{\infty}(G)$. In order to investigate this question, we will need the space ${\it D}$ of all dominant weights on  $G$ and we define ${\it D}_{0}: = D \setminus \{0\}$. Recall that these are in one--to--one correspondence with elements of $\G$, so that each $\lambda \in {\it D}$ is mapped to a unique $\pi_{\lambda} \in \G$. We will equip $\g$ with an Ad--invariant inner product, and write the associated norm as $|\cdot|$. This induces a norm on ${\it D}$ which is denoted by the same symbol. All results that follow are taken from \cite{Sug} (see also Chapter 3 of \cite{App1}). Writing $d_{\lambda}: = d_{\pi_{\lambda}}$, we have the useful estimates
\begin{equation} \label{est1}
d_{\lambda} \leq C_{1} |\lambda|^{m},
\end{equation}
where $C_{1} \geq 0$ and $m$ is the number of positive roots of $G$, and for all $X \in \g$,  there exists $C_{2} \geq 0$ so that
\begin{equation} \label{est2}
||d\pi_{\lambda}(X)||_{HS} \leq C |\lambda|^{\frac{m+2}{2}}|X|.
\end{equation}

We will also need {\it Sugiura's zeta function} $\zeta: \C \rightarrow \R \cup \{\infty\}$, defined by \begin{equation} \label{Sugz}
\zeta(s) = \ds\sum_{\lambda \in {\it D}_{0}}\frac{1}{|\lambda|^{2s}},
\end{equation}
which converges whenever $2\Re(s) > r$, where $r$ is the rank of $G$.
 Finally we will need the fact that there is a topological isomorphism between $C^{\infty}(G)$ and the Suguira space of rapidly decreasing functions on ${\it D}$. We won't need full details of this, only the following fact that, for ease of reference, we will state as a theorem:
 \begin{theorem} \label{Sugis}
The function $f \in C^{\infty}(G)$ if and only if for all $p \geq 0$,
$$\lim_{|\lambda| \rightarrow \infty}|\lambda|^{p}||\widehat{f}(\lambda)||_{HS} = 0,$$ where $\widehat{f}(\lambda):= \widehat{f}(\pi_{\lambda})$.
\end{theorem}

\begin{theorem} \label{niceest} There exist $K = K(|X_{1}|, \ldots, |X_{d}|) > 0$ so that for all $\lambda \in D$,
$$ \sup_{\sigma \in G}||j_{A}(\sigma, \lambda )||_{HS} \leq K(1 + |\lambda|^{m+2}),$$
where $j_{A}(\cdot, \lambda):= j_{A}(\cdot, \pi_{\lambda})$.
\end{theorem}

\begin{proof} Throughout this proof $K_{1}, K_{2}, \ldots $ are non--negative constants. We examine each of the four terms on the right hand side of (\ref{PDo2}) in turn. For the first of these, we use (\ref{est1}) to obtain for all $\lambda \in D$,
$$ ||I_{\pi_{\lambda}}||_{HS} = d_{\lambda}^{\frac{1}{2}} \leq K_{1}|\lambda|^{m/2}.$$
For the second term we apply (\ref{est2}). For the third term, we employ the equivalence of norms in finite--dimensional vector spaces to find that for all $X, Y \in \g$,
\bean ||d\pi_{\lambda}(X)d\pi_{\lambda}(Y)||_{HS} & \leq & K_{2} ||d\pi_{\lambda}(X)d\pi_{\lambda}(Y)||_{op}\\
& \leq & K_{2}||d\pi_{\lambda}(X)||_{op}||d\pi_{\lambda}(Y)||_{op}\\
& \leq & K_{3}||d\pi_{\lambda}(X)||_{HS}||d\pi_{\lambda}(Y)||_{HS}\\
& \leq & K_{4}|\lambda|^{m+2}|X||Y|. \eean
For the term controlled by $\mu$, let $U$ be a canonical neighbourhood of $e$ for which $x_{1}, \ldots, x_{d}$ are canonical co--ordinates. Then for each $\tau \in U$,
$$ \pi_{\lambda}(\tau) = \pi_{\lambda}\left(\exp{\left(\sum_{i=1}^{d}x_{i}(\tau)X_{i}\right)}\right) = \exp{\left(\sum_{i=1}^{d}x_{i}(\tau)d\pi_{\lambda}(X_{i})\right)},$$
where the final $\exp$ denotes the matrix exponential. Then by a standard Taylor series argument we have,
\bean & & \sup_{\sigma \in G}\left|\left|\int_{U}\left(\pi_{\lambda}(\tau) - I_{\pi_{\lambda}} - \sum_{i=1}^{d}x_{i}(\tau)d\pi_{\lambda}(X_{i})\right)\mu(\sigma, d\tau)\right|\right|_{HS} \\
& \leq & \sup_{\sigma \in G}\int_{U}\left|\left|\exp{\left(\sum_{i=1}^{d}x_{i}(\tau)d\pi_{\lambda}(X_{i})\right)} - I_{\pi_{\lambda}} - \sum_{i=1}^{d}x_{i}(\tau)d\pi_{\lambda}(X_{i})\right|\right|_{HS}\mu(\sigma, d\tau)\\
& \leq & \sup_{\sigma \in G}\int_{U}\sum_{i,j=1}^{d}x_{i}(\tau)x_{j}(\tau)||d\pi_{\lambda}(X_{i})d\pi_{\lambda}(X_{j})||_{HS}\mu(\sigma, d\tau)\\
& \leq & K_{5}\max\{|X_{1}|^{2}, \ldots, |X_{d}|^{2}\}|\lambda|^{m+2}\sup_{\sigma \in G}\int_{U}\sum_{i=1}^{d}x_{i}(\tau)^{2}\mu(\sigma, d\tau)\\
& \leq & K_{6}|\lambda|^{m+2}. \eean

Finally we have
\bean & & \sup_{\sigma \in G}\left|\left|\int_{U^{c}}\left(\pi_{\lambda}(\tau) - I_{\pi_{\lambda}} - \sum_{i=1}^{d}x_{i}(\tau)d\pi_{\lambda}(X_{i})\right)\mu(\sigma, d\tau)\right|\right|_{HS} \\
& \leq & \sup_{\sigma \in G}\int_{U^{c}}\left(2 d_{\lambda}^{\frac{1}{2}} + \sum_{i=1}^{d}|x_{i}(\tau)|||d\pi_{\lambda}(X_{i})||_{HS}\right)\mu(\sigma, d\tau)\\
& \leq & (2 d_{\lambda}^{\frac{1}{2}} + K_{7}\max_{i = 1, \ldots, d}\{||d\pi_{\lambda}(X_{i})||_{HS}\})\sup_{\sigma \in G}\mu(\sigma, U^{c})\\
& \leq & K_{8}(|\lambda|^{m/2} + |\lambda|^{(m+2)/2}), \eean
where we used the continuity of the map $g \rightarrow \mu(g, U^{c})$, as discussed after the proof of Theorem \ref{contA}.

The result follows on combining together all these estimates, bearing in mind that the functions $c, b_{i}$ and $a_{jk}$, for $1 \leq i,j,k \leq d$, are all bounded on $G$.

\end{proof}

\begin{lemma} \label{abuncon}
If $f \in C^{\infty}(G)$ then the series $\sum_{\lambda \in D}d_{\lambda}\tr(j_{A}(\sigma, \pi)\widehat{f}(\lambda)\pi_{\lambda}(\sigma))$ converges absolutely and uniformly in $\sigma \in G$.
\end{lemma}

\begin{proof} Using (\ref{est1}), Theorem \ref{niceest}, and the basic matrix estimate $|\tr(AB)| \leq ||A||_{HS}||B||_{HS}$, we have
\bean & &   \sum_{\lambda \in D}d_{\lambda}|\tr(j_{A}(\sigma, \pi)\widehat{f}(\lambda)\pi_{\lambda}(\sigma))| \\
& \leq & C_{1}K \sum_{\lambda \in D} |\lambda|^{m}(1 + |\lambda|^{m+2})||\widehat{f}(\lambda)||_{HS}. \eean

By Theorem \ref{Sugis}, we can find $\lambda_{0} \in D_{0}$ so that for all $|\lambda| > |\lambda_{0}|$,
$$   ||\widehat{f}(\lambda)||_{HS} \leq \frac{B_{1}}{|\lambda|^{2s}},$$
where $2s > 2(m + 1) +r$ and $B_{1} > 0$. Then there exists $B_{2} > 0$ so that
$$ \sup_{\sigma \in G}\sum_{|\lambda| > |\lambda_{0}|}d_{\lambda}|\tr(j_{A}(\sigma, \pi)\widehat{f}(\lambda)\pi_{\lambda}(\sigma))| \leq B_{2}\zeta(s - m-1) < \infty.$$
\end{proof}

\begin{theorem} The operator $A: C^{\infty}(G) \rightarrow C(G)$ is a pseudo--differential operator with symbol $j_{A}$.
\end{theorem}

\begin{proof}We must show that for all $f \in C^{\infty}(G), \sigma \in G$,
$$ Af(\sigma) = \sum_{\pi \in \G}d_{\pi}\tr(j_{A}(\sigma, \pi)\widehat{f}(\pi)\pi(\sigma)).$$
We have shown that this holds for $f \in {\mathcal E}(G)$. Now let $f \in C^{\infty}(G)$. By Theorem 3.3.1 (ii) of \cite{App1}, the Fourier series $\sum_{\lambda \in D}d_{\lambda}\tr(\widehat{f}(\lambda)\pi_{\lambda}(\sigma))$ converges absolutely and uniformly to $f(\sigma)$. Hence (by ordering the weights in $D$ e.g. by increasing norm, with arbitrary ordering for those having the same norm) we obtain a sequence $(f_{n})$ of partial sums in ${\mathcal E}(G)$ that converges uniformly to $f$. By lemma \ref{abuncon}, we have the uniform convergence, $$\lim_{n \rightarrow \infty}Af_{n}(\sigma) = \sum_{\pi \in \G}d_{\pi}\tr(j_{A}(\sigma, \pi)\widehat{f}(\pi)\pi(\sigma)),$$ and the result follows since $\overline{A}$ is closed.
\end{proof}

For examples of symbols associated to classes of Feller processes on compact Lie groups, see section 5 of \cite{App2}. We note that these include processes obtained by solving stochastic differential equations driven by L\'{e}vy processes, processes obtained by subordination, and a Lie group generalisation of Feller's pseudo--Poisson process. A simpler class of examples is obtained by taking $A$ to be the killed Hunt generator of Theorem \ref{Hu1}. Here the symbol (as we expect) is independent of $\sigma \in G$, and so only a function of $\pi \in \G$. Using (\ref{Huntf}) we easily compute
\bean j_{A}(\pi) & = & -cI_{\pi} + \sum_{i=1}^{d}b_{i}d\pi(X_{i}) + \frac{1}{2}\sum_{i,j=1}^{d}a_{ij}d\pi(X_{i})d\pi(X_{j})\\
& = & \int_{G}\left(\pi(\tau) - I_{\pi} - \sum_{i=1}^{d}x_{i}(\tau)d\pi(X_{i})\right)\mu(d\tau).
\eean
If we put $c=0$ we obtain the formula that was found for symbols of Hunt semigroups in \cite{App3}.

Using techniques developed in \cite{FR}, we conjecture that it should also be possible to represent operators satisfying the PMP as pseudo--differential operators when the group is homogeneous.

\vspace{5pt}

Acknowledgements. We are grateful to Ming Liao, and also to the referees, for helpful comments that have improved the quality of the paper.

\end{document}